%% file: main.tex
\documentclass[11pt,twoside,reqno]{amsart}
\usepackage[all]{xy}
        \usepackage {amssymb,latexsym,amsthm,amsmath,mathtools,dsfont}
        \usepackage{enumitem,color}

\usepackage{mathtools}

\usepackage{babel}

\usepackage{hyperref}
\long\def\symbolfootnote[#1]#2{\begingroup%
\def\thefootnote{\fnsymbol{footnote}}\footnote[#1]{#2}\endgroup}

\newcommand{\Sp}[2]{\ensuremath{\textup{Sp}({#1},\mathbb{F}_{q^{#2}})}}
\newcommand{\Un}[2]{\textup{U}({#1},\mathbb{F}_{q^{#2}})}
\newcommand{\fq}{\ensuremath{\mathbb{F}_q}}
\newcommand{\fqn}[1]{\mathbb{F}_{q^{#1}}}

\newcommand{\diag}{\textup{diag}}
\newcommand{\Mp}{$\widetilde{M}$\textup{-power }}
\newcommand{\GL}{\textup{GL}}

\newcommand{\On}[3]{\textup{O}^{#1}({#2},\mathbb{F}_{q^{#3}})}
\newcommand{\U}{\textup{U}}

\DeclareMathOperator{\ch}{Char}

\makeatletter
\def\imod#1{\allowbreak\mkern10mu({\operator@font mod}\,\,#1)}
\makeatother
\makeatletter
\renewcommand*\env@matrix[1][*\c@MaxMatrixCols c]{%
  \hskip -\arraycolsep
  \let\@ifnextchar\new@ifnextchar
  \array{#1}}
\makeatother
\usepackage[capitalize,nameinlink]{cleveref} 
\usepackage{comment} 
\usepackage{xcolor} 
\hypersetup{
    colorlinks,
    linkcolor={red!80!black},
    citecolor={green!80!black},
    urlcolor={blue!80!black}
}
\makeatletter
\def\Ddots{\mathinner{\mkern1mu\raise\p@
\vbox{\kern7\p@\hbox{.}}\mkern2mu
\raise4\p@\hbox{.}\mkern2mu\raise7\p@\hbox{.}\mkern1mu}}
\makeatother

\newtheorem{theorem}{Theorem}[section]
\newtheorem{question}{Question}[section]
\newtheorem{lemma}[theorem]{Lemma}
\newtheorem{corollary}[theorem]{Corollary}
\newtheorem{proposition}[theorem]{Proposition}
\newtheorem*{theorem*}{Theorem}
\theoremstyle{definition}
\newtheorem{definition}[theorem]{Definition}
\newtheorem{remark}[theorem]{Remark}

\numberwithin{equation}{section}
\newcommand{\ignore}[1]{}

\newcommand{\mynote}[1]{}
\begin{document}
\setcounter{section}{0}
\title{Powers in finite unitary groups}
\author[Panja S.]{Saikat Panja}
\email{panjasaikat300@gmail.com}
\address{Harish-Chandra Research Institute, Chhatnag Rd, Jhusi, Uttar Pradesh 211019, India}
\author[Singh A.]{Anupam Singh}
\email{anupamk18@gmail.com}
\address{IISER Pune, Dr. Homi Bhabha Road, Pashan, Pune 411 008, India}
\thanks{The first named author (Panja) is supported by PDF-M fellowship from HRI and the second named author (Singh) is funded by SERB through CRG/2019/000271 for this research.}
\date{\today}


\begin{abstract}
Let $\Un{n}{2}$ denote the subgroup of unitary matrices of the general linear group $\GL(n,\fqn{2})$ which fixes a Hermitian form and $M\geq 2$ an integer. This is a companion paper to the previous works \cite{KuSi22} and \cite{PaSi22} where the elements of the groups $\GL(n,\fq)$, $\Sp{2n}{}$, $\On{\pm}{2n}{}$ and $\On{}{2n+1}{}$ which has an $M$-th root in the concerned group, have been described. Here we will describe the $M$-th powers in unitary groups for the regular semisimple, semisimple and cyclic elements. Our methods are parallel to those of the Memoir \emph{``A generating function approach to the enumeration of matrices in classical groups over finite fields"} by Fulman, Neumann and Praeger.
\end{abstract}
\maketitle
\input{intro.tex}
\input{Conj.tex}
\input{polynomial.tex}
\input{separable.tex}
\input{cyclic.tex}
\input{semisimple.tex}

\end{document}

%% file: intro.tex
\section{Introduction}\label{sec:intro}
After the settlement of Ore's conjecture, \emph{Waring type problems} in group theory has emerged as an important area of research, posing new difficulties and beautiful results. This subject has a long history, which dates back to the work of Armand Borel, e.g. Borel dominance theorem. Given an element $w\in \mathcal{F}_l$ (the free group on $l$ generators), the map associated with $w$ by plugging elements of $G^l$ in $w$, is called a \emph{word map}. It was proved by A. Borel (in \cite{Bo83}) (and later by Larsen independently in \cite{La04}) that given a semisimple algebraic group $\mathcal{G}$ and a word map $w\colon \mathcal{G}^l\longrightarrow \mathcal{G}$, it is a dominant map ({that is, image is dense in $\mathcal{G}$ with respect to the Zariski topology}). The image of $w$ will be denoted as $w(\mathcal{G})$ hereafter as is the tradition in the subject. The result due to Borel reveals the surprising fact that $w(\mathcal{G})^2 = \mathcal{G}$. It is therefore natural to ask what happens when $\mathcal{G}$ is replaced by an arbitrary group. The general question in this direction is as follows.
\begin{question}
Let $G$ be a group and $w$ be a word. Then
\begin{enumerate}
    \item Is $w$ a surjective map on $G$?
    \item If $w$ is not a surjective map is there a constant $k_w\in\mathbb{N}$ such that $w(G)^{k_w}=\langle w(G)\rangle$.
\end{enumerate}
\end{question}
\noindent Several results have been proved in this direction in the last three decades. For example, it was shown first that for a finite simple group $G$, there exists a constant $c$ such that $G = w(G)^c$, where $w=x_1x_2x_1^{-1}x_2^{-1}\in\mathcal{F}_2$ (see \cite{Wi96} by J. Wilson). Later by collective effort, it was shown that $c$ can be chosen to be $1$. This is essentially Ore's conjecture. Around 1996, Martinez and Zelmanov proved that any element of a sufficiently large
simple group can be written as a product of $f(k)$ many $k$-th powers in \cite{MaZe96}. This was also independently proved by Saxl and Wilson in \cite{SaWi97}. This simple-looking result has far-reaching consequences. For example, in the paper \cite{SaWi97} the authors prove that if $G$ is a Cartesian product of nonabelian finite simple groups with $G$ being a finitely generated profinite group, then every subgroup $H$ of finite index in $G$ is open. 
Over ten years back A. Shalev proved that for a word $w$, there exists positive integer $N(w)$ such that every finite simple group $G$ having size more than $N(w)$ satisfies $w(G)^3=G$ (see \cite[Theorem 1.1]{Sh09}). Later this was improved to the result that $w(G)^2=G$ for sufficiently large finite simple groups of Lie type $G$, depending upon $w$, by Larsen and Larsen in \cite{LaSha09}. In fact, there are words $w$ for which $w(G)\neq G$, and one of the primary examples is a power map.

Fix a positive integer $M\geq 2$. In this article, we choose to focus on the word $x^M\in\mathcal{F}_1$.  The corresponding word map will be called \emph{power map}. Our main focus will be to compute probabilities of an element of a particular type to be $M$-th power. We call an element $g\in G$ to be $M$-th power if there exists $\alpha\in G$ such that $\alpha^M=G$. This has been of prime interest for quite some time now. We started looking at this through the approach of statistical group theory. This enables us to answer this question in terms of generating functions. This has been done in \cite{KuSi22} for finite general linear groups, and in \cite{PaSi22} for finite symplectic and orthogonal groups. We will answer this question for unitary groups here. We start with the 
description of conjugacy classes and centralizers of an element in \autoref{sec:conjugacy-centralizer}. This will be used to deduce the problem in the realm of counting conjugacy classes of a particular type, rather than counting elements of a particular type. In \autoref{sec:poly}
we define the $\widetilde{M}$-power polynomial which will help us
deduce which conjugacy classes are the $M$-th powers. Later on in \autoref{sec:sep}, \autoref{sec:cyc} and \autoref{sec:semisimple-matrices} we deduce the main results.
\subsection*{Acknowledgement} A part of this work has been completed during the first named author's visit to ISI Bangalore in the Fall of 2022. He acknowledges the support of ISI Bangalore during this visit.

%% file: Conj.tex
\section{Conjugacy classes and centralizer in unitary groups}\label{sec:conjugacy-centralizer}
This section briefly recalls the definitions and mentions the results about conjugacy classes and centralizers of elements in unitary groups. 
The exposition here is mainly based on the works \cite{wa63} by G. E. Wall and \cite{ta21} by D. E. Taylor.
For a prime $p$ consider $q=p^l$ for some $l\in \mathbb{N}$. 
Consider the field $k=\mathbb F_{q^2}$, and the map $\sigma \colon \fqn{2}\longrightarrow\fqn{2}$ defined as $\sigma(a):=\overline{a}=a^q$ which is an order two automorphism of the field. This further induces an automorphism of
$\fqn{2}[t]$, the polynomial ring over      $\fqn{2}$. The image of
$f\in\fqn{2}[t]$ will be denoted as $\overline{f}$. Consider the Hermitian form given by the matrix
\[
\Lambda_n=
\begin{pmatrix}
0 & 0 & \ldots & 0 & 1 \\
0 & 0 & \ldots & 1 & 0 \\
\vdots & \vdots & \Ddots & \vdots & \vdots \\
0 & 1 & \ldots & 0 & 0 \\
1 & 0 & \ldots & 0 & 0 
\end{pmatrix}.
\]
Then we identify the unitary group (corresponding to the above Hermitian form) with the set
\[
\U(n,\fqn{2})=\left\{A\in\GL(n,\fqn{2})\mid  A\Lambda_n\overline{A}^t=\Lambda_n\right\}.
\]
Since all non-degenerate Hermitian forms on a $k^n$ are unitarily congruent to the standard one that is the Hermitian form given by the above matrix,
we get that all the unitary groups arising are conjugate with each other inside
$\GL(n,\fqn{2})$ (see theorem $10.3$ of \cite{Gr02}). From the work of Wall, it is known that two elements $g_1, g_2 \in\U(n,\fqn{2})$ are conjugate to each other if and only if they are conjugate in $\GL(n,\fqn{2})$.
These conjugacy classes are parameterized by special polynomials and partitions.
We call a polynomial $f$ of degree $d$ to be \emph{$\sim$-symmetric} if $\widetilde{f}=f$ where $\widetilde{f}(t) = \overline{f(0)}^{-1} t^d \bar f(t^{-1})$. A polynomial without a proper $\sim$-symmetric factor will be called a \emph{$\sim$-irrdeucible} polynomial. Note that a $\sim$-irreducible polynomial can be reducible in the usual sense.
Denote by $\widetilde{\Phi}$ to be the set of all $\sim$-irreducible polynomials. 
Let $\mathcal{P}_{n}$ denote the set of all partitions of numbers $\leq n$. Then a conjugacy class in $\U(n,\fqn{2})$ is determined by a function $\lambda\colon \widetilde{\Phi} \longrightarrow \mathcal{P}_n$ which satisfies
\begin{enumerate}
    \item $\lambda(f)=\lambda(\widetilde{f})$,
    \item $\sum\limits_{\phi}|\lambda(\phi)|\deg\phi=n$.
\end{enumerate}

We recall the definitions of several kinds of matrices here, for which the existence of $M$-th root will be discussed in this article. A matrix $A$ is said to be \emph{separable} if its characteristic polynomial is separable over $\fqn{2}$, \emph{semisimple} when the minimal polynomial is separable. Call a matrix
$A$ to be cyclic if the minimal polynomial is same as the characteristic polynomial
of $A$.

%% file: polynomial.tex
\section{Special polynomials}\label{sec:poly}
For our work, we need to deal with certain kinds of polynomials which we describe now.
\begin{definition}
For an integer $M\geq 2$, a self-conjugate irreducible monic (SCIM) polynomial $f$ of degree $d$ will be called a \emph{$\widetilde{M}$-power polynomial} if $f(x^M)$ has a SCIM factor $g$ of degree $d$. 
\end{definition}
\begin{lemma}\label{lem:number-of-polynomials}
    Let $\widetilde{N}_{M}(q,d)$ denote the number of $\widetilde{M}$-power SCIM polynomials of degree $d$. Then we have 
\begin{align*}
\widetilde{N}_{M}(q,d)=\dfrac{1}{d(M,q^{2d}-1)}\sum\limits_{l|d}^{}\mu(l)\left(M(q^{2d/l}-1),q^d+1\right).\end{align*}
\end{lemma}
\begin{proof}
Let $f$ be an \Mp SCIM polynomial of degree $d$. Then $f(x^M)$ has a SCIM factor of degree $d$, say $g$.
Now consider $f,g\in\fqn{2d}[x]$. Then we get that
\begin{align*}
    f(x)=\prod\limits_{i=1}^d(x-\alpha_i)\,\text{ and }\,g(x)=\prod\limits_{i=1}^d(x-\beta_i).
\end{align*}
Note that we may assume that $\beta_i^M=\alpha_i$, after a permutation of the indices.
Consider the map $\theta_M\colon \fqn{2d}\longrightarrow\fqn{2d}$ given by $\theta_M(\delta)=\delta^M$. Then $\alpha_i\in\textup{Im}(\theta)$. 
Since $f$ is an SCIM polynomial we also have that $\alpha_i^{q^d+1}=1$.
Hence we have that
\begin{align*}
    \beta_i^{M(q^d+1)=\beta_i^{q^{2d}-1}=1},
\end{align*}
and that $\beta_i^M=\alpha_i$ generates $\fqn{2d}$ over $\fqn{2}$.

Conversely suppose $\alpha$ satisfies $\alpha^{q^d+1}=1$ and generates $\fqn{2d}$ over $\fqn{2}$. If $\varphi$ is the minimal polynomial of $\alpha$ over $\fqn{2}$, then 
it is of degree $d$. Also all roots $\eta$ of $\varphi$ satisfy that $\eta^{q^d+1}=1$. Thus if $\widetilde{N}_M(q,d)$ denotes the number of \Mp SCIM polynomials, we will have that
\begin{align*}
    \widetilde{N}_M(q,d)=\dfrac{1}{d}\left|\left\{\alpha\in\fqn{2d}\middle\vert\substack{\alpha^{q^d+1}=1,\alpha=\theta_M(\beta),\\\beta^{q^d+1}=1,\beta\in\fqn{2d},\fqn{2d}=\fqn{2}(\alpha)}\right\}\right|.
\end{align*}
Since $\lvert\theta_M^{-1}(1)\rvert=(M,q^{2d}-1)$, we get that
\begin{align*}
    \widetilde{N}_M(q,d)=\dfrac{1}{d(M,q^{2d}-1)}\left\lvert\left\{\beta\in\fqn{2d}\middle\vert \beta^{q^d+1}=1,\fqn{2d}=\fqn{2}(\beta^M)\right\}\right\rvert.
\end{align*}
To ensure that $\fqn{2d}=\fqn{2}(\beta^M)$, we must have that $\beta^M\not\in\fqn{d/l}$ for all $l|d$ where $l>1$. To count this, define
\begin{align*}
    E_l=\left\{\beta\in\fqn{2d}\middle\vert \beta^{q^d+1}=1,\fqn{2d/l}=\fqn{2}(\beta^M)\right\}.
\end{align*}
Then we have that $\lvert E_l\rvert=(M(q^{2d/l}-1),q^{d+1})$. Hence by the inclusion-exclusion principle, we get that
\begin{align*}
    \widetilde{N}_M(q,d)=\dfrac{1}{d(M,q^{2d}-1)}\sum\limits_{l|d}\mu(l)\left(M(q^{2d/l}-1),q^d+1\right)
\end{align*}
\end{proof}

%% file: separable.tex
\section{Separable matrices}\label{sec:sep}
\begin{lemma}\label{le:root-implies-m}
Let $f$ be an SCIM polynomial of odd degree $d\geq 1$ and $\alpha^M=C_f$. Then $f(x^M)$ has a SCIM factor of degree $d$.
\end{lemma}
\begin{proof}
Considering $C_f$ in $\Un{d}{2d}$ we have that $C_f\sim\diag(\lambda_1,\lambda_2,\ldots,\lambda_d)$, since it is semisimple. Given that $\alpha^M=C_f$, we can assume $\alpha$ to be semisimple.
Then $\alpha\sim\diag(\alpha_1,\alpha_2,\ldots,\alpha_d)$ where $\alpha_i$'s are roots of minimal polynomial of $\alpha$. Then $\alpha_i^M=\lambda_{j(i)}$. We may assume that $j$ is the identity function, after permutation.

Hence $f(\alpha_i^M)=0$. Considering $H(x)=f(x^M)$ we get that $H(\alpha_i)=0$. Hence $\min_{\fq}(\alpha)|H$. Since $g=\min_{\fq}(\alpha)|H$ is irreducible and $\alpha\in\Un{d}{q}$, we get that $g$ is SCIM and is the required polynomial.
\end{proof}
\begin{lemma}\label{lem:m-implies-root}
Let $f$ be a SCIM polynomial of odd degree $d$ and $f(x^M)$ has a SCIM factor $g$ of degree $d$. Then there exists $\alpha\in\Un{d}{2}$ such that $\alpha^M=C_f$.
\end{lemma}
\begin{proof}
    We will show that $C_g^M\sim C_f$. In $\fqn{2d}$ we have that 
    \begin{align*}
        f(x)&=\prod_{i=1}^{d}(x-\lambda_i)\\
        g(x)&=\prod_{i=1}^{d}(x-\alpha_i),
    \end{align*}
    for the appropriate choice of $\alpha_i$s and $\lambda_i$s. Since $g(x)|f(x^M)$, 
    we obtain that
    \begin{align*}
        \prod_{i=1}^{d}(\alpha_1^M-\lambda_i)=0.
    \end{align*}
    This implies $\alpha_1^M=\lambda_j$ for some $j$, which after permutation can be assumed to be $1$. If $h$ is characteristic polynomial of $C_g^M$, then $h(\lambda_1)=0$. So $f=h$ and this finishes the proof.
\end{proof}
\begin{proposition}\label{prop:sep-conj-class}
    Let $cS_U^M(n,q)$ denote the number of conjugacy classes of separable matrices in $\Un{n}{2}$ which are $M$-th power. Let $cS_U^M(z)=1+\sum\limits_{n=1}^\infty cS_U^M(n,q)$. Then we have that
    \begin{align*}
        cS_U^M(z)=\prod_{d=\text{odd}}(1+z^d)^{\widetilde{N}_M(q,d)}\prod_{d\geq 1}(1+z^{2d})^{\widetilde{R}_M(q,d)}.
    \end{align*}
\end{proposition}
\begin{proof}
    Let $A$ be a separable matrix and the combinatorial data attached with $A$ 
    is given by $\Delta_A=\{(f,\lambda_f)|f\in\Phi\setminus\{0\}\}$. 
    Since $A$ is separable, we must have that $\lambda_f\in\{0,1\}$. 
    Thus in view of \cref{le:root-implies-m} and \cref{lem:m-implies-root} we get that $X$ has an $M$-th root if and only if
    \begin{enumerate}
        \item for $f=\widetilde{f}$, we have $f\in\widetilde{\Phi}_M$,
        \item for $f\neq \widetilde{f}$, we have $f\in\Phi_M\setminus\widetilde{\Phi}_M$.
    \end{enumerate}
    Hence the characteristic polynomial of $X$ is given by
    \begin{align*}
        \ch_A(x)=\prod f\prod g\widetilde{g},
    \end{align*}
    for appropriate choice of $f=\widetilde{f}\in\widetilde{\Phi}_M$ and $g\neq \widetilde{g}\in\Phi_M\setminus\widetilde{\Phi}_M$, depending upon $A$. Thus using \cref{lem:number-of-polynomials}, we conclude that,
    \begin{align*}
        cS_U^M(z)&=\prod\limits_{f\in\widetilde{\Phi}_M}\left(1+z^{\deg{f}}\right)
        \prod\limits_{\substack{\{g,\widetilde{g}\}\\g\in\Phi_M\setminus\widetilde{\Phi}_M}}\left(1+z^{2\deg{g}}\right)\\
        &=\prod_{d=\text{odd}}\left(1+z^d\right)^{\widetilde{N}_M(q,d)}\prod_{d\geq 1}\left(1+z^{2d}\right)^{\widetilde{R}_M(q,d)}
    \end{align*}
\end{proof}
\begin{theorem}\label{th:probability-separable}
     Let $S_U^M(n,q)$ denote the number of separable matrices in $\Un{n}{2}$ which are $M$-th power. Let $S_U(z)=1+\sum\limits_{n=1}^\infty S_U^M(n,q)$. Then we have that
    \begin{align*}
        S_U^M(z)=\prod_{d=\text{odd}}\left(1+\dfrac{z^d}{q^d+1}\right)^{\widetilde{N}_M(q,d)}\prod_{d\geq 1}\left(1+\dfrac{z^{2d}}{q^{2d}-1}\right)^{\widetilde{R}_M(q,d)}.
    \end{align*}
\end{theorem}
\begin{proof}
    We will use the work of Wall (see \cite{wa63}) for the computation of the corresponding probability.
    Given a matrix $A$ with $\ch_A(x)=\prod_{i,j}f_ig_j\widetilde{g_j}$ for appropriate choices of $f_i,g_j$, we have that the centralizer of $A$ is of order
    \begin{align*}
        C_{\Un{n}{2}}(A)=\prod_{i}C_{\Un{n_i}{2}}(X_i)\prod_{j}C_{\Un{2m_j}{2}}(Y_{j}),
    \end{align*}
    for $X_i$ being element of $\Un{n_i}{2}$ with $\ch_{X_i}(x)=f_i$ and $Y_i\in\Un{2m_i}{2}$ satisfies $\ch_{Y_j}(x)=g_j\widetilde{g_j}$. Since (See \cite{fnp05}, page 32)
    \begin{align*}
        |C_{\Un{n_i}{2}}(X_i)|=q^{n_i}+1\text{ and }|C_{\Un{2m_j}{2}}(Y_{j})|=q^{2m_j-1},
    \end{align*}
    we get the desired result.
\end{proof}

%% file: cyclic.tex
\section{Cyclic matrices}\label{sec:cyc}
    We will denote the matrix $\begin{pmatrix}
        C_f&I&0&\cdots&0&0\\
        0&C_f&I&\cdots&0&0\\
        \vdots&\vdots&\vdots&\ddots&\vdots&\vdots\\
        0&0&0&\cdots&C_f&I\\
        0&0&0&\cdots&0&C_f
    \end{pmatrix}$ by $U(f,n)$ where $C_f$ denotes the standard companion matrix of $f$ in $\Un{\deg f}{2}$. We will start with the following easy but crucial observation.
\begin{lemma}\label{lem:unitary-power}
The matrix $U(f,n)^M$ is conjugate to the matrix $U(g,n)$ where $U(f,1)^M\sim U(g,1)$.
\end{lemma}

\begin{proposition}\label{prop:power-cyclic-conjugacy}
    Let $cC_U^M(n,q)$ denote the number of conjugacy classes of cyclic matrices in $\Un{n}{2}$ which are $M$-th power. Let $cC_U^M(z)=1+\sum\limits_{n=1}^\infty cC_U^M(n,q)z^n$. Then we have that
    \begin{align*}
        cC_U^M(z)=&\prod\limits_{d=\text{ odd}}\left(1-z^d\right)^{-\widetilde{N}_M(q,d)}
        \times\prod\limits_{d\geq 1}\left(1-z^{2d}\right)^{-\widetilde{R}_M(q,d)}
    \end{align*}
\end{proposition}
\begin{proof}
    In the primary decomposition of a given cyclic matrix $A$, there is a unique summand corresponding to each irreducible factor of $\ch_A(x)$. If $\phi$, an irreducible factor of $\ch_A(x)$ satisfies $\widetilde{\phi}=\phi$, then the block corresponding to $\phi$ is an $M$-th power 
    if and only if $\phi$ is $\widetilde{M}$-power polynomial.
    Again if $\phi\neq \widetilde{\phi}$, then the block corresponding to $\phi\widetilde{\phi}$ is an $M$-th power if and only if $\phi$ is an $M$-power polynomial. 

    Hence if the combinatorial data attached with $A$ 
    is given by $\Delta_A=\{(f,\lambda_f)|f\in\Phi\setminus\{0\}\}$, then $A$ will be an $M$-th power if and only if
    \begin{enumerate}
        \item for $f=\widetilde{f}$, we have $f\in\widetilde{\Phi}_M$,
        \item for $f\neq\widetilde{f}$, we have $f\in\Phi_M$.
    \end{enumerate}
    Hence we get that
    \begin{align*}
        cC_U^M(z)&=\prod\limits_{f\in\widetilde{\Phi}_M}\left(1+z^{\deg f}+z^{2\deg f}+\ldots\right)\prod\limits_{\substack{\{g,\widetilde{g}\}\\g\in\Phi_M\setminus\widetilde{\Phi}_M\}}}\left(1+z^{2\deg f}+z^{4\deg f}+\ldots\right)\\
        &=\prod\limits_{d=\text{ odd}}\left(1-z^d\right)^{-\widetilde{N}_M(q,d)}
        \times\prod\limits_{d\geq 1}\left(1-z^{2d}\right)^{-\widetilde{R}_M(q,d)}
    \end{align*}
\end{proof}
\begin{theorem}\label{thm:power-cyclic-element}
    Let $C_U^M(n,q)$ denote the number of cyclic matrices in $\Un{n}{2}$ which are $M$-th power. Let $C_U^M(z)=1+\sum\limits_{n=1}^\infty C_U^M(n,q)z^n$. Then we have that
    \begin{align*}
        C_U^M(z)=&\prod\limits_{d=\text{odd}}\left(1+\dfrac{z^d}{\left(q^d+1\right)\left(1-\left(\dfrac{z}{q}\right)^d\right)}\right)^{\widetilde{N}_M(q,d)}\\
        \times&\prod\limits_{d\geq 1}\left(1+\dfrac{z^{2d}}{\left(q^{2d}-1\right)\left(1-\left(\dfrac{z}{q}\right)^{2d}\right)}\right)^{\widetilde{R}_M(q,d)}
    \end{align*}
\end{theorem}
\begin{proof}
    As discussed in \cref{th:probability-separable}, similarly here in this case centralizer of a matrix $A$, is direct product of
    with one term for each SCIM polynomial dividing the $\ch_A(x)$ and each pair $\{g,\widetilde{g}\}$ such that
    $g\widetilde{g}$ dividing $\ch_A(x)$. Note that if $f=\widetilde{f}$ is irreducible of degree $d$ and $f(x)^m|\ch_A(x)$, then the term corresponding to this factor is given by $q^{d(m-1)}(q^d+1)$ (see \cite{fnp05}, page 38).
    The term corresponding to the pair $\{g,\widetilde{g}\}$ is $q^{2d(m-1)}(q^{2d}-1)$ where $\left(g(x)\widetilde{g}(x)\right)^m|\ch_A(x)$ and $g$ is of degree $d$. Then we have that
    \begin{align*}
        C^M_U(z)&=\prod\limits_{d=\text{ odd}} \left(1+\dfrac{z^d}{q^d+1}+\dfrac{z^{2d}}{q^d(q^d+1)}+\dfrac{z^{3d}}{q^{2d}(q^d+1)}+\ldots\right)^{\widetilde{N}_M(q,d)}\\
        &\times\prod\limits_{d\geq 1}\left(1+\dfrac{z^{2d}}{q^{2d}-1}+\dfrac{z^{4d}}{q^{2d}(q^{2d}-1)}+\dfrac{z^{6d}}{q^{4d}(q^{2d}-1)}+\ldots\right)^{\widetilde{R}_M(q,d)}
    \end{align*}
\end{proof}

%% file: semisimple.tex
\section{Semisimple matrices}\label{sec:semisimple-matrices}
In this section we will concentrate on the case where $M$ is a prime satisfying $(M,q)=1$.
This can be generalized with conscious effort in case of prime power (as was done in Section $6$ of \cite{KuSi22})
or general $M$ satisfying $(M,q)=1$ using indicator functions (see Section $6$ of \cite{PaSi22}). Thus to avoid cumbersome expressions, we choose to work with prime $M$. We will need the following result, which dates back to 1955.
\begin{lemma}[Section $3$, \cite{Bu55}]\label{factorization-m-power}
    Suppose that the polynomial $f(x)$, of degree $n$, is irreducible
over $\fq$ and that $t$ is the order of its root. Let $m$ be a natural number such that $(m,q)=1$ and $m=m_1m_2$, where $(m_1,t)=1$ and each prime factor of $m_2$ is a divisor of $t$. Then
\begin{enumerate}
    \item each root of $f(x^m )$ has order of the form $dm_2t$, where $d$ is a divisor
of $m_1$;
\item if $d|m_1$ , then $f(x^m )$ has exactly $nm_2\phi(d)/\mathcal{M}(dm_2t,q)$ irreducible
factors over $\fq$ with roots of order $dm_2t$ and each of these factors has degree
$\mathcal{M}(dm_2t,q)$.
\end{enumerate}
Here $\mathcal{M} (s, q)$ is the least natural number satisfying $q^{\mathcal{M}(s,q)}\equiv 1\pmod{s}$, if the integer $s$ is relatively prime to $q$.
\end{lemma}
\begin{corollary}\label{block-size-semisimple}
    Let $M$ be a prime satisfying $(M,q)=1$. Assume $f$ to be an irreducible monic polynomial of degree $d$ over $\fqn{2}$. Then 
    \begin{enumerate}
        \item either $f$ is an $M$-power polynomial
        \item or $f(x^M)$ is an irreducible polynomial.
    \end{enumerate}
\end{corollary}
Note that if the set of all $\widetilde{M}$-power polynomials are subset of $M$-power polynomials.
Consider $\mathcal{S}=\Phi\setminus\Phi_M$. This set consists of two kinds of polynomials as follows:
\begin{enumerate}
    \item polynomial $f$ such that $f=\widetilde{f}$ but it is not an $M$-th power,
    \item polynomial $f$ such that $f\neq\widetilde{f}$ and is not an $M$-th power.
\end{enumerate}
The number of polynomial of first kind of fixed degree $d$ will be denoted by $\widetilde{S}'_M(d,q)$ and the number of pair of polynomials $\{g,\widetilde{g}\}$ of second kind of degree $d$ will be denoted by $S'_M(d,q)$. Note that
\begin{align*}
    S'_M(d,q)=\widetilde{N}(q,d)-\widetilde{N}_M(q,d).
\end{align*}
\begin{proposition}\label{prop:semisimple-conjugacy}
    Let $cS_U^M(n,q)$ denote the number of conjugacy classes of semisimple matrices in $\Un{n}{2}$ which are $M$-th powers. Let $cS_U^M(z)=1+\sum\limits_{n=1}^\infty cS^M_U(n,q)z^n$. Then we have that
    \begin{align*}
        cS_U^M(z)=&\prod\limits_{d=\text{ odd }}\left(1-z^d\right)^{-\widetilde{N}_M(q,d)}\times\prod\limits_{d\geq 1}\left(1-z^{2d}\right)^{-\widetilde{R}_M(q,d)}\\
        \times&\prod\limits_{d=\text{ odd }}\left(1-z^{dM}\right)^{-\widetilde{S}'_M(d,q)}\times 
        \prod\limits_{d\geq 1}\left(1-z^{2dM}\right)^{-S'_M(d,q)}.
    \end{align*}
\end{proposition}
\begin{proof}
        Let $A$ be a semisimple matrix and the combinatorial data attached with $A$ 
    is given by $\Delta_A=\{(f,\lambda_f)|f\in\Phi\setminus\{0\}\}$. Here we encounter with four types of polynomials, as described below:
    \begin{enumerate}
        \item $f$ is an $\widetilde{M}$-power SCIM polynomial,
        \item $f$ is an ${M}$-power polynomial, but satisfies $f\neq\widetilde{f}$,
        \item $f$ is an SCIM polynomial but neither $M$-power nor $\widetilde{M}$-power,
        \item anything other than the above three classes.
    \end{enumerate}
    Then we will have the following constituents coming from each of the factors occurring in the factorization of the characteristic polynomial. Note that
    \begin{enumerate}
        \item if $f\in \widetilde{\Phi}_M$ then $\lambda_f\in\mathbb{Z}_{\geq 0}$,
        \item since $f\neq \widetilde{f}$ is an $M$ power polynomial, we get that $\lambda_f=\lambda_{\widetilde{f}}\in\mathbb{Z}_{\geq 0}$,
        \item if $f$ is an SCIM polynomial but not $\widetilde{M}$-power polynomial, then by \cref{block-size-semisimple} we get that $\lambda_f\in M\mathbb{Z}_{\geq 0}$,
        \item in the last case we get that $f\neq \widetilde{f}$ and $f(x^M)$ is irreducible, hence we get that in this case as well $\lambda_f=\lambda_{\widetilde{f}}\in M\mathbb{Z}_{\geq 0}$.
    \end{enumerate}
    Hence putting all the cases together we get that
    \begin{align*}
        cS_U^M(z)=&\prod\limits_{f\in\widetilde{\Phi}_M}\left(1+z^{\deg f}+z^{2\deg f}+\cdots\right)
        \times \prod\limits_{\substack{\{f,\widetilde{f}\}\\f\in\Phi_M}}\left(1+z^{2\deg f}+z^{4\deg f}+\cdots\right)\\
        \times &\prod\limits_{\substack{f\in\widetilde{\Phi}\setminus\widetilde{\Phi}_M}}\left(1+z^{M\deg f}+z^{2M\deg f}+\cdots\right)\\
        \times &\prod\limits_{\substack{\{f,\widetilde{f}\}\\f\in\Phi\setminus\{\widetilde{\Phi}\cup\Phi_M\}}}\left(1+z^{2M\deg f}+z^{4M\deg f}+\cdots\right)\\
        =&\prod\limits_{d=\text{ odd }}\left(1-z^d\right)^{-\widetilde{N}_M(q,d)}\times\prod\limits_{d\geq 1}\left(1-z^{2d}\right)^{-\widetilde{R}_M(q,d)}\\
        \times&\prod\limits_{d=\text{ odd }}\left(1-z^{dM}\right)^{-\widetilde{S}'_M(d,q)}\times 
        \prod\limits_{d\geq 1}\left(1-z^{2dM}\right)^{-S'_M(d,q)}.
    \end{align*}
\end{proof}
\begin{theorem}\label{thm:M-power-semisimple}
    Let $S_U^M(n,q)$ denote the number of semisimple matrices in $\Un{n}{2}$ which are $M$-th powers. Let $S_U^M(z)=1+\sum\limits_{n=1}^\infty S^M_U(n,q)z^n$. Then we have that
    \begin{align*}
        S_U^M(z)=&\prod\limits_{d=\text{ odd }}\left(1+\sum\limits_{m\geq 1}\dfrac{z^{dm}}{|\U(m,q^{2d})|}\right)^{\widetilde{N}_M(q,d)}
        \times\prod\limits_{d\geq 1}\left(1+\sum\limits_{m\geq 1}\dfrac{z^{2dm}}{|\GL(m,q^{2d})|}\right)^{\widetilde{R}_M(q,d)}\\
        \times&\prod\limits_{d=\text{ odd }}\left(1+\sum\limits_{m\geq 1}\dfrac{z^{dmM}}{|\U(mM,q^{2d})|}\right)^{\widetilde{S}'_M(d,q)}
        \times
        \prod\limits_{d\geq 1}\left(1+\sum\limits_{m\geq 1}\dfrac{z^{2dmM}}{|\GL(mM,q^{2d})|}\right)^{S'_M(d,q)}.
    \end{align*}
\end{theorem}
\begin{proof}
    Note that centralizer order of $A$ with combinatorial data $(f,\lambda_f)$ where $f\in\widetilde{\Phi}$ is given by $|\U(\lambda_f,q^{2\deg f})|$ if $f=\widetilde{f}$ and for the pair $\{f,\widetilde{f}\}$ is given by $|\GL(\lambda_f,q^{2\deg f})|$. Hence the result follows from \cref{prop:semisimple-conjugacy}.
\end{proof}